\pdfoutput=1
\documentclass[a4paper,10pt]{article}
\usepackage[utf8]{inputenc}
\usepackage[T1]{fontenc}
\usepackage{amsmath}
\usepackage{amsthm}
\usepackage{amssymb}
\usepackage{stmaryrd}
\usepackage[backend=biber,style=alphabetic,sorting=nty,giveninits]{biblatex}
\usepackage{hyperref}

\newcommand{\Z}{\mathbb{Z}}
\newcommand{\Q}{\mathbb{Q}}
\newcommand{\R}{\mathbb{R}}
\newcommand{\FF}{\mathcal{F}}
\newcommand{\DualizingComp}{D\underline{\Q}}
\newcommand{\EquivEulerChar}[1][n]{\chi^{\mathfrak{S}_{#1}}}
\newcommand{\support}{\textrm{supp}}
\newcommand{\trace}{\textrm{Tr}}

\theoremstyle{plain}
\newtheorem{thm}{Theorem}[section]
\newtheorem{introthm}{Theorem}

\newtheorem{introcor}[introthm]{Corollary}
\newtheorem{lem}[thm]{Lemma}
\newtheorem{prop}[thm]{Proposition}

\newtheorem{cor}[thm]{Corollary}

\theoremstyle{definition}
\newtheorem{defi}[thm]{Definition}

\theoremstyle{remark}
\newtheorem{rmk}[thm]{Remark}

\title{The Euler characteristic of configuration spaces}
\author{\uppercase{Louis Hainaut}}%\footnote{The author gratefully acknowledges support by ERC-2017-STG 759082}}

\begin{document}
	
	\maketitle
	
	\abstract{
		In this short note we present a generating function computing the compactly supported Euler characteristic $\chi_c(F(X, n), K^{\boxtimes n})$ of the configuration spaces on a topologically stratified space $X$, with $K$ a constructible complex of sheaves on $X$, and we obtain as a special case a generating function for the Euler characteristic $\chi(F(X, n))$. We also recall how to use existing results to turn our computation of the Euler characteristic into a computation of the equivariant Euler characteristic.
		
%		\vspace{0.5cm}
		
%		keywords: configuration space of points, Euler characteristic, topologically stratified space
		
		\vspace{0.5cm}
		
	}
	
	\section{Introduction}
	We consider (ordered) configuration spaces of $n$ distinct points on a topological space $X$, for $n\geq 0$ an integer:
	\[
	F(X, n) = \{(x_1, \ldots, x_n)\in X^n| x_i\neq x_j \mbox{ for } i\neq j\}
	\]
	These spaces admit a free action by the symmetric group $\mathfrak{S}_n$ by permutation of the coordinates. One can also consider the quotient space
	\[
	B(X, n) = F(X, n)/{\mathfrak{S}_n}.
	\]
	In particular, the map $F(X, n) \to B(X, n)$ is a covering map.
	
	We will be interested in computing the Euler characteristic of the spaces $F(X, n)$. In the case of a manifold $M$, the exponential generating function for the Euler characteristic of configuration spaces on $M$ is
	\begin{equation}
		\sum_{n=0}^{\infty}{\chi(F(M, n))\cdot\frac{t^n}{n!}} = \left\{\begin{array}{ll}
			(1+t)^{\chi(M)} & \mbox{if $\dim(M)$ is even}\\
			(1-t)^{-\chi(M)} & \mbox{if $\dim(M)$ is odd}.
		\end{array}\right. \label{formula_manifold}
	\end{equation}
	
	This formula can be found using the fibration
	\[
	M\setminus\{p_1, \ldots, p_{n-1}\} \to F(M, n) \to F(M, n-1).
	\]
	Using multiplicativity of the Euler characteristic on fibrations (and Mayer-Vietoris to compute the characteristic of the fibers) we can prove inductively that
	\[
	\chi(F(M, n)) = \left\{\begin{array}{ll}
		\prod_{i=0}^{n-1}{(\chi(M) - i)} & \mbox{if $\dim(M)$ is even}\\
		\prod_{i=0}^{n-1}{(\chi(M) + i)} & \mbox{if $\dim(M)$ is odd}.
	\end{array}\right.
	\]
	and these coefficients induce the aforementioned exponential generating functions. %(A reference is nice, but not needed.)
	
	We are therefore primarily interested in the case when $X$ is not a manifold. If $X$ is a finite simplicial complex, the following formula was found by Gal \cite{Gal01}:
	\begin{equation} \label{formula_Gal}
		\sum_{n=0}^{\infty}{\chi(F(X, n)\cdot\frac{t^n}{n!}} = \prod_{\sigma}{(1 + (-1)^{d_{\sigma}}(1-v_{\sigma})t)^{(-1)^{d_{\sigma}}}}
	\end{equation}
	where the product runs over all the cells $\sigma$ of $X$, $d_{\sigma}$ denotes the dimension of $\sigma$ and $v_{\sigma}$ the Euler characteristic of its normal link $L_{\sigma}$. If we use the compactly supported Euler characteristic, Getzler \cite{Get1} found the following formula, which applies to any locally compact Hausdorff space:
	\begin{equation*}
		\sum_{n=0}^{\infty}{\chi_c(F(X, n))\cdot \frac{t^n}{n!}} = (1+t)^{\chi_c(X)}.
	\end{equation*}
	This formula will be further discussed in Lemma \ref{euler_series_trivial_coeff}. We can already notice that it generalizes formula \eqref{formula_manifold}, since for even-dimensional manifolds $M$ we have $\chi_c(M) = \chi(M)$, while for odd-dimensional ones we have $\chi_c(M) = -\chi(M)$.
	
	The formula we present here generalizes \eqref{formula_Gal} to any topologically stratified space of finite type $X = \cup_{\alpha}{X_{\alpha}}$.
	
	\begin{introthm}
		Let $X = \cup_{\alpha}{X_{\alpha}}$ be a topologically stratified space of finite type. For any constructible complex $K$ of sheaves on $X$ we have
		
		\[
			\sum_{n=0}^{\infty}{\chi_c(F(X, n), K^{\boxtimes n})\cdot \frac{t^n}{n!}} = \prod_{\alpha}{(1 + \chi(K|_{X_{\alpha}})\cdot t)^{\chi_c(X_{\alpha})}},
		\]
		with $\chi(K|_{X_{\alpha}})$ denoting the alternating sum of the ranks of the local system $K|_{X_{\alpha}}$.
	\end{introthm}

	\begin{introcor}
		Specializing the previous theorem to the case where $K$ is the dualizing complex of $X$, we obtain
		\begin{equation}\label{new_formula_special_case}
			\sum_{n=0}^{\infty}{\chi(F(X, n)\cdot\frac{t^n}{n!}} = 
			\prod_{\alpha}{(1 + (-1)^{d_{\alpha}}(1 - \chi(L_{\alpha}))\cdot t)^{ (-1)^{d_{\alpha}}\chi(X_{\alpha})}},
		\end{equation}
		 where $d_{\alpha}$ denotes the dimension of the stratum $X_{\alpha}$ and $L_{\alpha}$ its link.
	\end{introcor}
	
	We remark that in the special case when $X$ is a finite simplicial complex, with the stratification given by its open cells, our formula \eqref{new_formula_special_case} becomes the same as Gal's \eqref{formula_Gal}.
	
	\begin{rmk}
		Many "reasonable" spaces admit a topological stratification. For example every algebraic variety, or more generally every semialgebraic or subanalytic set admits a topological stratification, since they all admit a Whitney stratification.
	\end{rmk}
	
	Although the formula \ref{new_formula_special_case}, to the best of our knowledge, was not known when we started this project, we learned during the final stages of writing this paper that an essentially equivalent result is contained in a paper of Baryshnikov \cite{Barysh}\footnote{The full proof is not contained in the published paper due to page constraints, but can be found in the full version of the paper available from the author's webpage \cite{BarWeb}}.
	
	Nevertheless, the argument presented here may have some independent value: the proof given here seems more direct, and it only uses formal properties of sheaf cohomology and Grothendieck's six functors (whereas Baryshnikov's arguments involve "fattening" the diagonals and a "cubulation" of the space). In particular it is valid also for e.g. $\ell$-adic étale cohomology of arbitrary varieties.
	
	Another difference is that Baryshnikov works with a more general class of "exotic" configuration spaces, in which collisions of particles are allowed according to a certain set of rules. Our arguments could be generalized to cover such spaces, too, although we have not done so here.
	
	\subsection{Conventions} \label{section:setup}
	In the following we will use \emph{topologically stratified spaces of finite type}. We recall the definition of topologically stratified spaces \cite{GM83} and some of their properties.
	
	\begin{defi}
		An $n$-dimensional topologically stratified space is a Hausdorff space $X$ with a filtration
		\[
		\emptyset = X_{-1} \subset X_0 \subset X_1 \subset \ldots \subset X_n = X
		\]
		of $X$ by closed subspaces such that for each $i$ and each $x\in X_i\setminus X_{i-1}$ there exists a neighborhood $U_x\subset X$ of $x$, a compact $(n-i-1)$-dimensional topologically stratified space $L_x$ and a filtration preserving homeomorphism
		\[
		U_x \cong \R^i \times CL_x
		\]
		with $CL_x$ denoting the open cone on $L_x$.
		
		The connected components $X_{\alpha}$ of the set differences $X_i\setminus X_{i-1}$ are called the \emph{strata} of $X$.
	\end{defi}
	
	\begin{rmk}
		From the definition of topologically stratified spaces we can derive following useful properties:
		\begin{itemize}
			\item It is not required that any of the inclusions $X_{i-1}\subset X_i$ are strict. In particular every $n$-dimensional topologically stratified space admits an $(n+1)$-dimensional topological stratification with $X_{n+1} = X_n$.
			\item Since the homeomorphism $U_x\cong \R^i\times CL_x$ is filtration preserving, the strata $X_{\alpha}$ are topological manifolds.
			\item We call the space $L_x$ the \emph{link} of $x$. We can prove that on any stratum $X_{\alpha}$ we can choose the same link for all points $x\in X_{\alpha}$ and we will therefore write this space $L_{\alpha}$.
		\end{itemize}	
	\end{rmk}
	
	We say that a topologically stratified space $X$ is \emph{of finite type} if the collection of strata is finite and every stratum is (homeomorphic to) the interior of a compact manifold with boundary. This latter condition is enough to ensure that the strata have the homotopy type of a finite CW-complex, and in particular that their Betti numbers are finite.
	
	All sheaves considered are sheaves of $\Q$-vector spaces, so in particular all sheaves are flat. We will always assume that the complexes of sheaves are cohomologically bounded and that local systems have finite rank. We will also need the notion of a \emph{constructible} complex of sheaves. A sheaf $\FF$ on a stratified space $X = \cup_{i=1}^m{X_i}$ is called \emph{constructible} if its restriction to each stratum $X_i$ is a local system of finite rank. A complex of sheaves $\FF^{\bullet}$ is called \emph{constructible} if each cohomology sheaf $\mathcal{H}^p(\FF^{\bullet})$ is constructible. 
	
	We will use $\chi_c(X, \FF^{\bullet})$ the compactly supported Euler characteristic with sheaf coefficients. As always, the Euler characteristic is defined as the alternating sum of the ranks of the cohomology groups, so we need to know what is $\mathbb{H}_c^{\bullet}(X, \FF^{\bullet})$, with $\FF^{\bullet}$ being a complex of sheaves.
	
	\begin{defi}
		Let $\FF$ be a sheaf on a space $X$. We define the functor of global sections of $\FF$ with compact support as
		\[
		\Gamma_c(X, \FF) = \{s\in \Gamma(X, \FF)\colon \support(s) \mbox{ is compact}\},
		\]
		the support of a section $s\in \Gamma(U, \mathcal{F})$ being the set
		$\{x\in U : s_x \neq 0\}$.
	\end{defi}
	
	As for the regular sheaf cohomology, we then define the cohomology groups with compact supports $\mathbb{H}_c^{\bullet}(X, \FF^{\bullet})$ to be the higher direct image of the functor $\Gamma_c(X, -)$. When $\FF^{\bullet}$ is the constant sheaf $\underline{\Q}$ concentrated in degree $0$, we may omit the coefficient.
	
	\begin{rmk} $\ $
		\begin{itemize}
			\item When $X$ is compact, then $\Gamma_c(X, -) = \Gamma(X, -)$ and therefore for any complex of sheaves $\FF^{\bullet}$ we have $\mathbb{H}^{\bullet}(X, \FF^{\bullet}) = \mathbb{H}_c^{\bullet}(X, \FF^{\bullet})$. This means in particular that $\chi(X, \FF^{\bullet}) = \chi_c(X, \FF^{\bullet})$.
			\item When $X$ is a manifold of dimension $d$, we have $\chi_c(X) = (-1)^d\chi(X)$, by Poincaré duality.
		\end{itemize}
	\end{rmk}

	\paragraph*{Acknowledgment.}
	I would like to thank my PhD supervisor, Dan Petersen, for his invaluable support during this project.% I also thank Robin Stoll for helpful conversations while I was writing Section \ref{Sec:exotic}, as well as Valentin Imbach who helped me figure out the combinatorial arguments for this section.
	I gratefully acknowledge financial support by ERC-2017-STG 759082.
	
	\section{Computing the Euler characteristic}
	
	Although the introduction focused on the Euler characteristic $\chi(X)$, we will work with the Euler characteristic with compact support $\chi_c(X, \FF)$. We explain later how the former is a special case of the latter.
	
	\begin{defi}
		Let $X$ be a topological space and $K$ be a complex of sheaves on $X$. We define the generating series
		\[
		e(X, K; t) = \sum_n{\chi_c(F(X, n), K^{\boxtimes n})\cdot \frac{t^n}{n!}}.
		\]
		
		This expression makes sense if all the compactly supported Euler characteristics on the right hand side exist.
	\end{defi}
	
	We will use the result that the Euler characteristic with compact supports $\chi_c(X, \FF)$ is additive over stratifications. We present here a proof of this result for the reader's convenience.
	
	\begin{lem}\label{compact_euler_additive}
		Let $X = \cup_{i=1}^m{X_{i}}$ be a space with a finite stratification and $K$ be a complex of sheaves on $X$. Then we have
		\[
		\chi_c(X, K) = \sum_{i=1}^{m}{\chi_c(X_i, K)}
		\]
	\end{lem}
	\begin{proof}
		The key result is that whenever $X$ is partitioned into two subsets $U$ and $Z$ with $U$ open and $Z$ closed, then
		\begin{equation}\label{compact_euler_open_close}
			\chi_c(X, K) = \chi_c(U, K) + \chi_c(Z, K).
		\end{equation}
		More information about this identity can be found in \cite[Remark 2.4.5]{Dim04}
		
		We prove Lemma \ref{compact_euler_additive} by induction on the number of strata. This result is true whenever the (finite) stratification satisfies the two conditions that every stratum is locally closed and that the closure of each stratum is a union of strata. These conditions are clearly satisfied in the case of a topological stratification.
		\begin{itemize}
			\item If the stratification has one stratum there is nothing to prove.
			\item If the stratification has two or more strata, by the conditions mentioned above we can assume without loss of generality that the stratum $X_1$ is closed, and let $X'$ be the union of all remaining strata. The identity \eqref{compact_euler_open_close} gives that $\chi_c(X, K) = \chi_c(X_1, K) + \chi_c(X', K)$, and we conclude by employing the induction hypothesis on $X'$.
		\end{itemize}
	\end{proof}
	
	If $L$ is a complex of sheaves whose cohomology sheaves are local systems, we define its Euler characteristic $\chi(L)$ as the alternating sum
	\[
	\chi(L) = \sum_{s\in \Z}{(-1)^s rk(\mathcal{H}^s(L))}.
	\]
	The ranks are well-defined since the cohomology sheaves are local systems.
	
	\begin{rmk}\label{euler_characteristic_multiplicative}
		If $L, L'$ both have cohomology sheaves which are local systems, then $\chi(L\otimes L') = \chi(L)\chi(L')$ by the K{\"u}nneth theorem for complexes of sheaves.
	\end{rmk}
	
	\begin{lem} \label{euler_characteristic_local_system}
		Let $X$ be the interior of a compact manifold with boundary and let $L$ be a complex of sheaves on $X$ whose cohomology sheaves are local systems. Then $\chi_c(X, L) = \chi_c(X)\chi(L)$.
	\end{lem}
	
	\begin{proof}
		We first consider the special case when $L$ is a single local system in degree $0$. In that case $\chi(L) = rk(L)$ and we can prove the result as follows: Since $\chi_c(X, -) = (-1)^{\dim(X)}\chi(X, -)$, we can work with the usual cohomology. Moreover $X$ is homotopy equivalent to a compact manifold with boundary (its closure), and these spaces have the homotopy type of a finite CW-complex, so $X$ is homotopy equivalent to a finite CW-complex. We can therefore reduce the statement to the case when $X$ is a finite CW-complex, and this case can be proven by using cellular cochains (see also \cite[Proposition 2.5.4]{Dim04}).
		
		For the general case we use the spectral sequence relating sheaf cohomology and hypercohomology to obtain
		\begin{align*}
			\chi_c(X, L) &= \sum_i{(-1)^i rk(\mathbb{H}_c^i(X, L))} \\
			&= \sum_{p, q}{(-1)^{p+q} rk(H_c^p(X, \mathcal{H}^q(L)))} \\
			&= \sum_q {(-1)^q \chi_c(X, \mathcal{H}^q(L))} \\
			&= \chi_c(X) \chi(L).
		\end{align*}	
	\end{proof}
	
	\begin{lem}\label{euler_series_trivial_coeff}
		Let $X$ be a Hausdorff topological space. Then
		\[
		e(X, \underline{\Q}; t) = (1 + t)^{\chi_c(X)}
		\]
	\end{lem}
	
	\begin{proof}
		This formula is already known. We give an argument for the reader's convenience.
		
		The space $X^n$ admits a stratification whose strata are the configuration spaces $F(X, T)$ for $T$ any partition of the set $\{1, 2,\ldots n\}$. The configuration space $F(X, T)$ is defined as the space
		\[
		F(X, T) = \{ (x_1, \ldots, x_n) \in X^n| x_i = x_j \Leftrightarrow i\sim_T j\}
		\]
		or in words: $F(X, T)$ is the subspace of $X^n$ where two coordinates $x_i$ and $x_j$ are equal if and only if $i$ and $j$ belong to the same block of $T$. By Lemma \ref{compact_euler_additive} we obtain
		\[
		\chi_c(X^n) = \sum_{T\in \Pi_n}{\chi_c(F(X, T))}.
		\]
		
		Now on the one hand by multiplicativity we have $\chi_c(X^n) = \chi_c(X)^n$, and on the other hand, if $|T|$ denotes the number of blocks of the partition $T$, we have an identity $\chi_c(F(X, T)) = \chi_c(F(X, |T|))$. Therefore the above equality becomes
		\[
		\chi_c(X)^n = \sum_{k=1}^{n}{S(n, k)\chi_c(F(X, k))}
		\]
		with $S(n, k)$ a Stirling number of the second kind, defined as the number of partitions of $n$ in $k$ parts. Using (signed) Stirling numbers of the first kind $s(k, n)$, we can invert the relation between $\chi_c(X)^n$ and $\chi_c(F(X, k))$ to obtain
		\[
		\chi_c(F(X, k)) = \sum_{n=1}^{k}{s(k, n)\chi_c(X)^n}.
		\]
		
		Finally we use the generating function for the Stirling numbers of the first kind to obtain
		\[
		\sum_{k=0}^{\infty}{\chi_c(F(X, k))\frac{t^k}{k!}} = 1 + \sum_{k=1}^{\infty}{\sum_{n=1}^{k}{s(k, n)\chi_c(X)^n\frac{t^k}{k!}}} = (1 + t)^{\chi_c(X)}.
		\]
	\end{proof}
	
	\begin{lem}\label{euler_series_local_system}
		Let $X$ be the interior of a compact manifold with boundary and let $L$ be a complex of sheaves whose cohomology sheaves are local systems. Then
		\[
		e(X, L; t) = (1 + \chi(L)\cdot t)^{\chi_c(X)}.
		\]
	\end{lem}
	
	\begin{proof}
		We first notice that $\chi_c(F(X, n), L^{\boxtimes n}) = \chi_c(F(X, n))\cdot \chi(L)^n$ due to Lemma \ref{euler_characteristic_local_system} and Remark \ref{euler_characteristic_multiplicative}.
		%	Using this equality, we can then compute
		%	\begin{align*}
		%		e(X, L; t) &= \sum_n{\chi_c(F(X, n), L^{\boxtimes n})\cdot \frac{t^n}{n!}} \\
		%		&= \sum_n{\chi_c(F(X, n))\cdot \chi(L)^n\cdot \frac{t^n}{n!}} \\ 
		%		&= \sum_n{\chi_c(F(X, n))\cdot\frac{(\chi(L)\cdot t)^n}{n!}} \\
		%		&= e(X, \underline{\Q}; \chi(L)\cdot t).
		%	\end{align*}
		Then using Lemma \ref{euler_series_trivial_coeff} we obtain
		\[
		e(X, L; t) = e(X, \underline{\Q}; \chi(L)\cdot t) = (1 + \chi(L)\cdot t)^{\chi_c(X)}.
		\]
	\end{proof}
	
	\begin{prop}\label{product_stratification}
		Let $X = \bigcup_{j=1}^m{X_j}$ be a stratified space and $K$ be a complex of sheaves on $X$. Then
		\[
		e(X, K; t) = \prod_{j=1}^{m}{e(X_j, K|_{X_j}; t)}.
		\]
	\end{prop}
	
	\begin{proof}
		The stratification $(X_j)$ induces a stratification of $F(X, n)$ with strata of the form
		\[
		\prod_{j=1}^{m}{F(X_j, n_j)}
		\]
		such that $n_1 + \ldots + n_m = n$ and $n_j\geq 0$ for each $j$. Moreover, the restriction of $K^{\boxtimes n}$ to each such stratum is
		\[
		\boxtimes_{j=1}^m{(K|_{X_j})^{\boxtimes n_j}}.
		\]
		
		Once these two observations have been made, we can first compute for a fixed $n$
		\begin{align*}
			&\chi_c(F(X, n), K^{\boxtimes n})\cdot \frac{t^n}{n!} \\
			= &\sum_{n_1 +\ldots + n_m = n}{\binom{n}{n_1, \ldots, n_m}\chi_c\left(\prod_{j=1}^m{F(X_j, n_j), \boxtimes_{j=1}^m{(K|_{X_j})^{\boxtimes n_j}}}\right)\cdot \frac{t^n}{n!}} \\
			= &\sum_{n_1 +\ldots + n_m = n} \prod_{j=1}^m{\chi_c(F(X_j, n_j), (K|_{X_j}})^{\boxtimes n_j})\cdot \frac{t^{n_j}}{n_j!}
		\end{align*} 
		and thus we obtain for the exponential generating function
		\begin{align*}
			e(X, K; t) &= \sum_n{\chi_c(F(X, n), K^{\boxtimes n})\cdot \frac{t^n}{n!}} \\
			%&= \sum_n{\sum_{n_1 +\ldots + n_m = n}{\binom{n}{n_1, \ldots, n_m}\chi_c\left(\prod_{j=1}^m{F(X_j, n_j), \boxtimes_{j=1}^m{(K|_{X_j})^{\boxtimes n_j}}}\right)\cdot \frac{t^n}{n!}}} \\
			%&= \sum_n{\sum_{n_1 +\ldots + n_m = n} \prod_{j=1}^m{\chi_c(F(X_j, n_j), (K|_{X_j}})^{\boxtimes n_j})\cdot \frac{t^{n_j}}{n_j!}}\\
			&= \prod_{j=1}^m{\sum_n{\chi_c(F(X_j, n), (K|_{X_j})^{\boxtimes n})\cdot \frac{t^n}{n!}}}\\
			&= \prod_{j=1}^m{e(X_j, K|_{X_j}; t)}.
		\end{align*}
	\end{proof}
	
	We are now ready to state the main result:
	
	\begin{thm} \label{formula_euler_characteristic}
		Let $X = \cup_{i=1}^m{X_i}$ be a topologically stratified space of finite type and $K$ be a constructible complex of sheaves on $X$. Then
		\[
		e(X, K; t) = \prod_{i=1}^m{(1 + \chi(K|_{X_i})\cdot t)^{\chi_c(X_i)}}
		\]
	\end{thm}
	
	\begin{proof}
		From Proposition \ref{product_stratification}, we obtain that
		\[
		e(X, K; t) = \prod_{i=1}^m{e(X_i, K|_{X_i}; t)}.
		\]
		
		Then, for each stratum $X_i$ we apply Lemma \ref{euler_series_local_system} to obtain
		\[
		\prod_{i=1}^m{e(X_i, K|_{X_i}; t)} = \prod_{i=1}^m{(1 + \chi(K|_{X_i})\cdot t)^{\chi_c(X_i)}}.
		\]
	\end{proof}
	
	We now explain how to recover the Euler characteristics $\chi(F(X, n))$ from previous theorem.
	
	We define the dualizing complex $\DualizingComp$ of $X$ as the sheafification of the presheaf of cochain complexes
	\[
	U \mapsto C_c^{\bullet}(U, \Q)^{\vee}
	\]
	sending $U$ to the dual of the complex of cochains with compact support on $U$. Using \cite[Proposition V.2.4]{Ive86}, we see that this definition corresponds to the more usual one with the exceptional inverse image functor. A standard result of Verdier duality is that $\mathbb{H}_c^{\bullet}(X, \DualizingComp) = \mathbb{H}^{-\bullet}(X, \underline{\Q})^{\vee}$ (see \cite[Theorem 3.3.10]{Dim04}, using the fact that $D(\DualizingComp)\cong \underline{\Q}$). This means in particular that $\chi_c(X, \DualizingComp) = \chi(X)$. We also notice that for a topologically stratified space $X$, the dualizing complex has constructible cohomology since for any stratum $X_{\alpha}$ of dimension $i$ and any point $x\in X_{\alpha}$, $x$ admits a basis of neighborhoods all homeomorphic to $\R^i\times CL_{\alpha}$. Finally we easily see from the definition that $\DualizingComp$ is cohomologically bounded.
	
	The remarks made in the previous paragraph actually allow us to prove that for any stratum $X_{\alpha}$ and any point $x\in X_{\alpha}$ we have
	\begin{equation*}
		\mathcal{H}^{-p}(\DualizingComp|_{X_{\alpha}})|_x \cong H_c^p(\R^i\times CL_{\alpha})^{\vee}
	\end{equation*}
	and therefore 
	\begin{align*}
		\chi(\DualizingComp|_{X_{\alpha}}) &= \chi_c(\R^i\times CL_{\alpha}) \\
		&= (-1)^i \chi_c(CL_{\alpha}) \\
		&= (-1)^i (1 - \chi(L_{\alpha})).
	\end{align*}
	
	\begin{cor}\label{formula_usual_euler_characteristic}
		Specializing the Theorem \ref{formula_euler_characteristic} with $K = \DualizingComp$ gives the formula for the Euler characteristic
		\begin{align*}
			\sum_n{\chi(F(X, n))\cdot \frac{t^n}{n!}} &= \prod_{\alpha}{(1 + \chi(\DualizingComp|_{X_{\alpha}})\cdot t)^{\chi_c(X_{\alpha})}}\\
			&= \prod_{\alpha}{(1 + (-1)^{d_{\alpha}}(1 - \chi(L_{\alpha}))\cdot t)^{(-1)^{d_{\alpha}}\chi(X_{\alpha})}}.
		\end{align*}
	\end{cor}
	
	\subsection{The equivariant Euler characteristic}
	Let us remark that knowing the Euler characteristic $\chi(F(X, n))$ also determines $\chi(B(X, n))$ as well as $\EquivEulerChar(F(X, n))\in R(\mathfrak{S}_n)$. The content of this short section is probably well-known to the experts, but we decided to include it for completeness. The main observation that we will use here is that for any positive integer $n$, the space $F(X, n)$ admits a free action of the symmetric group $\mathfrak{S}_n$ by permutation of the coordinates.
	
	\begin{prop}
		Let $X$ be a topologically stratified space. Then
		\[
		\chi(B(X, n)) = \chi(F(X, n))/n!
		\]
		and
		\[
		\EquivEulerChar(F(X, n)) = \chi(B(X, n)) \cdot \Q[\mathfrak{S}_n].
		\]
		
		In particular, the equivariant Euler characteristic of $F(X, n)$ is a multiple of the regular representation.
	\end{prop}
	
	This result is a special case of a more general result about the Euler characteristic of spaces with a free action of a finite group, which seems to have been first written down by Zarelua \cite[Theorem 1]{Zar68}. The main ingredient of the proof is the Lefschetz trace formula, which implies that a trace $\trace(g) = \sum_i{(-1)^i \trace(g| H^i(F(X, n)))}$ is $0$ for all $g\in \mathfrak{S}_n$, $g\neq e$. The rest of the proof uses standard results of character theory on the double sum $$\sum_i{\sum_{g\in \mathfrak{S}_n}{(-1)^i \trace(g|H^i(F(X, n)))}}.$$
	
	\subsection{Example}
	We finish by showing how to use our formula for a space to which the previously known formulas do not apply. Note that in the proof we assumed that the strata where connected, but in computations we can allow non-connected strata as long as the link $L_x$ is the same for each point of the stratum. This doesn't change the result because of Lemma \ref{compact_euler_additive}.
	
	Consider the subspace $X\subset \R^3$ formed by two planes $\Pi_1, \Pi_2$ intersecting on a line $l$. The space $X$ is obviously not a manifold, and it also cannot be represented by a finite simplicial complex since it is non-compact, so Gal's formula does not apply.
	
	We can take the stratification $X = X_1\cup X_2$, with $X_1$ being the disjoint union of the four half-planes, and $X_2$ being the intersection line. In this situation, the link $L_1$ is the empty space, while the link $L_2$ is a discrete space with four points. We therefore obtain
	\[
	e(X, \DualizingComp; t) = (1+t)^4\cdot (1 + (-1)(1-4)\cdot t)^{-1} = (1+t)^4\cdot (1 + 3t)^{-1}.
	\]
	
\printbibliography

\end{document}